\documentclass[DIV=calc, paper=letter, fontsize=11pt]{scrartcl}	 

\usepackage[]{units}
\usepackage{lipsum} 
\usepackage[english]{babel} 
\usepackage[protrusion=true,expansion=true]{microtype} 
\usepackage{amsmath,amssymb,amsfonts,amsthm} 

\usepackage[svgnames, table]{xcolor} 
\usepackage[hang, small,labelfont=bf,up,textfont=it,up]{caption} 
\usepackage{booktabs} 
\usepackage{fix-cm}	 
\usepackage[]{units}
\usepackage{rotating}
\usepackage{tabularx} 
\usepackage{setspace}
\usepackage[colorlinks=true, linkcolor=blue, citecolor=blue, filecolor=blue, runcolor=blue, urlcolor = blue]{hyperref}

\usepackage{fullpage}
\usepackage{sectsty} 

\usepackage{fancyhdr} 
\usepackage{lastpage} 

\usepackage{nicefrac}
\usepackage{cite}

\usepackage{tikz}
\usetikzlibrary{calc,patterns,decorations.pathmorphing,decorations.markings}
\usepackage{pgfplots}

\usepackage{multicol}
\usepackage{float} 

\newcounter{tempcolnum}
\makeatletter
\newcommand{\multicolinterrupt}[1]{
\setcounter{tempcolnum}{\col@number}
\end{multicols}
#1%
\begin{multicols}{\value{tempcolnum}}
}
\makeatother

\theoremstyle{definition}
\newtheorem{definition}{Definition}[section]

\theoremstyle{theorem}
\newtheorem{theorem}{Theorem}[section]

\theoremstyle{proposition}
\newtheorem{proposition}{Proposition}[section]

\theoremstyle{lemma}
\newtheorem{lemma}{Lemma}[section]

\theoremstyle{corollary}
\newtheorem{corollary}{Corollary}[section]

\theoremstyle{conjecture}
\newtheorem{conjecture}{Conjecture}[section]

\theoremstyle{remark}
\newtheorem{remark}{Remark}[section]

\usepackage{graphicx}
\DeclareGraphicsExtensions{.png}

\usepackage{lettrine} 

\usepackage{caption}
\usepackage{graphicx}

\usepackage{titling} 

\newcommand{\HorRule}{\color{DarkGoldenrod} \rule{\linewidth}{1pt}} 

\pretitle{\vspace{-80pt} \begin{flushleft} \HorRule \fontsize{18}{18} \color{DarkRed} \selectfont} 

\title{\LARGE On the neighborhood of knots}

\posttitle{\par\end{flushleft}} 
\author{Eleni Panagiotou\\
School of Mathematical and Statistical Sciences, Arizona State University\\
Eleni.Panagiotou@asu.edu} 

\definecolor{issuePJA_color}{rgb}{1.0,0.0,0.0}

\definecolor{commentPJA_color}{rgb}{1.0,0.0,0.8}

\definecolor{commentEP_color}{rgb}{1.0,0.0,0.8}

\DeclareGraphicsExtensions{.png}

\begin{document}

\maketitle 

\thispagestyle{fancy} 
\begin{abstract}
    This manuscript introduces a new framework for the study of knots by exploring the neighborhood of knot embeddings in the space of simple open and closed curves in 3-space. The latter gives rise to a knotoid spectrum, which determines the knot type via its knot-type knotoids. We prove that the pure knotoids in the knotoid spectra of a knot, which are individually agnostic of the knot type, can distinguish knots of Gordian distance greater than one. We also prove that the neighborhood of some embeddings of the unknot can be distinguished from any embedding of any non-trivial knot that satisfies the cosmetic crossing conjecture.  Topological invariants of knots can be extended to their open curve neighborhood to define continuous functions in the neighborhood of knots. We discuss their properties and prove that invariants in the neighborhood of knots may be able to distinguish more knots than their application to the knots themselves. For example, we prove that an invariant of knots that fails to distinguish mutant knots (and mutant knotoids), can distinguish them by their neighborhoods, unless it also fails to distinguish non-mutant pure knotoids in their spectra. Studying the neighborhood of knots opens the possibility of answering questions, such as if an invariant can detect the unknot, via examining possibly easier questions, such as whether it can distinguish height one knotoids from the trivial knotoid.
\end{abstract}

MSC 57K10, 57K12

\section{Introduction}

Knots are simple closed curves in 3-space classified by  topological equivalence \cite{knot-book,Rolfsen2003}. It is known that the topology of a knot imposes constraints on geometrical aspects of the knot. For example, the ropelength, stick number, Mobius energy of a knot, all capture geometric and topological aspects of knots \cite{Cantarella2002,Freedman1994,Negami1991}. By thinking of knot embeddings as elements in the space of simple open and closed curves in 3-space endowed with a notion of distance, we can determine the neighborhood of a knot embedding.  We propose a new method for capturing geometric and topological properties of knots, by studying the space of open curves in their neighborhoods.

Knots and links are usually studied via diagrams, without need to refer to any particular embedding. On the other hand, the characterization of open curves in 3-space is inherently 3-dimensional and dependent on the embedding, since an open curve in 3-space can be continuously deformed to any other conformation. Diagrams of open curves can be characterized by knotoids (a classification of open arc diagrams) \cite{Turaev2012}, yet, no one knotoid can in general describe an open curve in 3-space. Rail approximations of open curves, closure approximations and knotoid approximations of open curves in 3-space have been proposed \cite{Turaev2012, Sulkowska2012,Goundaroulis2017}. The latter method proposes a knotoid spectrum (similar to the knot closure spectrum introduced in \cite{Millett2005}) for an open curve from which a dominant knotoid is extracted. An accurate description of an open curve in 3-space can be obtained as a superposition of knotoids with associated geometric probabilities \cite{Panagiotou2021,Panagiotou2020b,Barkataki2022}. In that way, the rigorous definition of measures of topological complexity of open curves in 3-space has been possible \cite{Panagiotou2021,Panagiotou2020b}.  These are continuous functions in the space of configurations of simple open curves in 3-space that tend to  topological invariants as the endpoints of the curve tend to coincide.

In this manuscript we define the neighborhood of a knot embedding in the space of open curves in 3-space. The $h$-neighborhood of a knot embedding consists of open curves at distance less than $h$ from the knot. Via the knotoid spectrum of every open curve in the neighborhood of a knot embedding, we can associate a knotoid spectrum to the knot embedding. The open neighborhoods of knots are sensitive both on the embedding and on the knot type. For any embedding of a knot, its neighborhood determines the knot type via the knot-type knotoid in its spectra. More importantly, we prove that any two embeddings of knots of Gordian distances greater than one can be distinguished by the pure knotoids in their neighborhood knotoid spectra, which are individually agnostic of the knot type. We also prove that there exist embeddings of the unknot whose pure knotoid spectra can be distinguished from that of any embedding of any non-trivial knot that satisfies the cosmetic crossing conjecture. Measures of entanglement at the neighborhood of knots give a topological and geometrical characterization of the knot that may be able to distinguish more knot types than its application to the knots themselves. We prove that an invariant that cannot distinguish mutant knots and mutant knotoids, can distinguish the neighborhoods of mutant knot embeddings, unless it also fails to distinguish non-mutant pure knotoids in their spectra. We also prove that an invariant that cannot distinguish the unknot, it can distinguish at least some embeddings of the unknot from any embedding of a non-trivial knot that satisfies the cosmetic crossing conjecture, unless it fails to distinguish height one knotoids from the trivial knotoid. These results point to a new approach towards answering questions in knot theory via their neighborhood characteristics.

The manuscript is organized as follows: Section 2 contains background information. Section 3 introduces the neighborhood of a knot and its properties.  Section 4 discusses the properties of knot invariants in the open curve neighborhood of a knot. Section 5 presents the conclusions of this study.

\section{Knots, knotoids and virtual knots}

This section gives an overview of the basic definitions regarding knots, knotoids, virtual knots and their properties \cite{Kauffmanvirtual1999, Gugumcu2017}.

\begin{definition}(knot)
    A knot is a simple closed curve in 3-space. Knots are classified under the notion of topological equivalence. A knot diagram can be thought of as a projection of a knot embedding, which keeps the information of over/under at the double points (crossings). Knots can be classified by studying their diagrams, which are classified under Reidemeister moves and planar isotopy \cite{knot-book} (see Figure \ref{vr1}).
\end{definition} 

\begin{definition}[knotoids, knot-type knotoids, pure knotoids]
    Knotoids are  classes of open  ended  knot  diagrams (see Figure \ref{fig:knotoids}) \cite{Turaev2012}. A knotoid diagram is a particular instance of a knotoid. The three  \textit{Reidemeister  moves} of knots  are  defined  on knotoid  diagrams by  modifying  the  diagram within small surrounding disks that do  not  utilize  the  endpoints (see Figures \ref{vr1} and \ref{fig:knotoids}).  
Two knotoid diagrams are said to be equivalent if they are related to each other by a finite sequence of such moves (and isotopy of $S^2$, $\mathbb{R}^2$ for knotoid diagrams in $S^2$, $\mathbb{R}^2$, respectively). A knot-type knotoid is a knotoid for which there is an end-to-end closure arc that does not intersect the rest of the diagram. A knotoid that is not knot-type is a pure knotoid. 
\end{definition}

A 3D interpretation of knotoids is possible via a rail construction \cite{Turaev2012}. Namely, consider two line segments perpendicular to the plane of the knotoid diagram that intersect its endpoints. By allowing the endpoints slide on the respective lines, respecting uncrossability, one obtains a 3D rail construction that corresponds to the same knotoid class. One can close up the parallel lines at two antipodal points away from the diagram of the knotoid to obtain a theta graph and study the equivalence class of theta graphs or the over/under closures of the knotoid \cite{Turaev2012}. 

\begin{figure}[ht!]
   \begin{center}
     \includegraphics[width=0.75\textwidth]{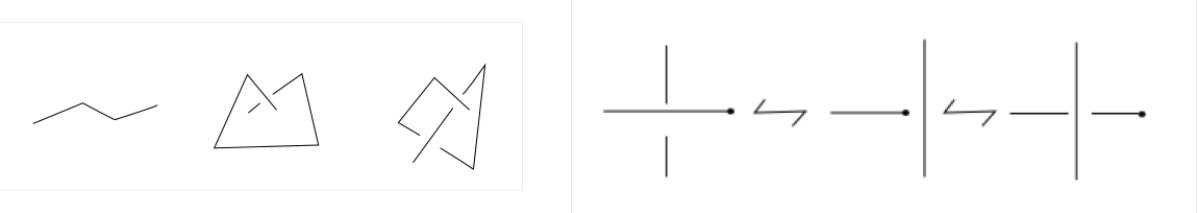}
      \end{center}
     \caption{(Left) Examples of (polygonal) knotoids (open simple arc diagrams). Notice that knotoids refer to \textit{projections of open chains}, while knots refer to closed chains in 3-space. (Right) Forbidden moves on knotoids. Knotoids are classified via Reidemeister moves and the forbidden moves.}
     \label{fig:knotoids}
\end{figure}

\begin{definition}(\textit{Virtual knot/link} and \textit{virtual knot/link diagram})
A \textit{virtual knot/link diagram} consists of generic closed curves in
$\mathbb{R}^2$ (or $S^2$) such that each crossing is either a classical crossing with over and under arcs, or
a virtual crossing without over or under information. Virtual knot/link diagrams are classified using the generalized Reidemeister moves, which include the classical Reidemeister moves and the virtual Reidemeister moves (See Figure \ref{vr1}). A \textit{virtual knot/link} is defined as an equivalence class of virtual knot/link diagrams under the generalised Reidemeister moves \cite{Kauffmanvirtual1999}. 
 \label{virtual_knot/link}
\end{definition}

Virtual knots also have a 3D interpretation. A representation of a virtual link, denoted $(F, L)$,
is an embedding of the link $L$ in $F \times I$ where $F$ is a closed, two dimensional,
oriented surface modulo Dehn twists, isotopy of the link with in $F \times I$, and
handle addition/subtraction \cite{Kauffmanvirtual1999,Kauffmanvirtual2001}. It was proved that every stable equivalence class of links in thickened surfaces has
a unique irreducible representative \cite{Kuperberg2003}.

\begin{figure}[ht]
    \centering
\includegraphics[scale=0.75]{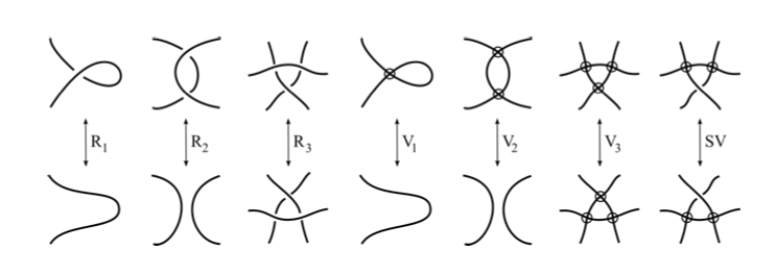}
    \caption{Generalised Reidemeister moves on virtual knots/links : the classical Reidemeister moves, $R_1$, $R_2$,
$R_3$; the virtual moves, $V_1$, $V_2$, $V_3$; and the semi-virtual move $SV$. (Figure from \cite{Barkataki2024}.)}
    \label{vr1}
\end{figure}

Knotoids and linkoids are connected to virtual knots and links via a closure of knotoids/linkoids with arcs that introduce only virtual crossings. Let $K$ be the set of knotoids and let $VK$ be the set of virtual knots and let $\phi:K\rightarrow VK$, denote the map that sends a knotoid to a virtual knot via a closure arc. Then $\phi$ is not surjective nor injective \cite{Turaev2012,Gugumcu2020}. Similar results exist for linkoids \cite{Barkataki2024,Adams2022}.

\begin{definition}[height, virtual crossing number]
     We will call the minimim number of crossings of the end-to-end closure arc in a knotoid diagram, its diagrammatic height. The height of a knotoid is the minimal diagrammatic height over all of its diagrams. 
    Similarly, in terms of virtual knots, the end-to-end closure arc will create a diagrammatic virtual crossing number. The virtual crossing number of a virtual knot is the minimum of diagrammatic virtual crossing numbers over all of its diagrams.
\end{definition}

\begin{remark}
    All knotoids correspond via virtual closure to knots that can be embedded in the thickenned torus. Indeed, any end-to-end closure arc of a knotoid can be represented in a thickenned surface  by adding a handle that follows the closure arc over all arcs involved in the crossing, giving an embedding in a thickenned torus. See also \cite{Korablev2017}.
\end{remark}

\begin{remark}
  Knotoids are also related to strongly invertible knots via a double branched covering \cite{Barbensi2022,Adams2022}.  
\end{remark}

\subsection{Essential secants of knots}

The secants of a knot embedding can help characterize the knotoid spectra in its open curve neighborhood, which is defined in the next section. We present the necessary definitions and some of the results that will be used regarding secants of knots \cite{Denne2006,Cantarella2003,Detoffoli2013}.

\begin{definition}[Essential and strongly essential secant, n-secant]
Let $\alpha, \beta, \gamma$ be three disjoint simple arcs from $p$ to $q$, forming a knotted virtual $\Theta$-graph, and consider the knot $K=\alpha\cup\gamma$ and suppose that $\beta$ is a straight segment, so that it can be thought of as a secant of the knot $K$ connecting the points $p, q$. 
Let $X:=\mathbb{R}\setminus(\alpha\cup\gamma)$ and consider a parallel curve $\delta$ to $\alpha\cup\beta$ in $X$ (meaning that $\alpha\cup\beta$ and $\delta$ cobound an annulus embedded in $X$). We choose the parallel to be homologically trivial in $X$. Let $x_0\in\delta$ near $p$ be a basepoint for $X$, and let $h=h(\alpha,\beta,\gamma)\in\pi_1(X,x_0)$ be the homotopy class of $\delta$. Then $(\alpha,\beta,\gamma)$ is inessential if $h$ is trivial. A secant arc A is inessential if either $\alpha$ or $\beta$ are homotopically trivial in $X$, otherwise it is essential. A secant arc is essential when both $\alpha$ and $\beta$ are not homotopically trivial. Now let $\lambda$ be a meridian loop (linking $\alpha\cup\gamma$ near $x_0$) in the knot complement. If the
commutator $[h(\alpha,\beta,\gamma),\lambda]$ is nontrivial then we say $(\alpha,\beta,\gamma)$ is strongly essential. The definition of (strongly) essential secant of a knot in $R^3$ can be associated to a crossing of a knot diagram. Namely, a projection of $\gamma$ with respect to a secant $\beta$, leads to a crossing in the diagram that we call (strongly) essential if the corresponding secant is (strongly) essential. 

An $n$-secant $a_1a_2\dotsc a_n$ is essential if we have $a_ia_{i+1}$ essential for each $i$ such that one of the arcs $\gamma_{a_ia_{i+1}}$ and $\gamma_{a_{i+1}a_i}$ includes no other $a_j$. An alternating quadrisecant is an essential 4-secant, where the order of the points on the secant are $a_{i-1}a_{i+1}a_ia_{i+2}$, where $a_{i-1]}, a_i, a_{i+1}, a_{i+2}$ is the order of the points along the knot. An alternating quadrisecant is essential if the middle secant $a_{i+1}a_i$ is essential \cite{Denne2006,Kuperberg2003,Pannwitz1933}. 

\end{definition}

It has been proved that every non-trivial knot has an essential quadrisecant \cite{Kuperberg1994} and similarly for an essential alternating quadrisecant \cite{Denne2006}. It is also proved that any diagram of a non-trivial knot has a strongly essential crossing \cite{Detoffoli2013}.

\begin{definition}
    A crossing in a knot diagram is nugatory if there is a circle in the
sphere of the diagram that intersects the diagram transversally at that crossing and does
not have any other intersection points with the diagram. Changing this crossing does not change the knot type. A non-nugatory crossing whose change does not change the knot type is called a cosmetic crossing.
\end{definition}

The following Conjecture is attributed to \cite{Lin2012} and has been proven to be true for several classes of knots \cite{Scharlemann1989,Kalfagianni2012,Torisu1999,Balm2012,Balm2016,Lidman2017}. 

\begin{conjecture} [Cosmetic crossing conjecture]
    If $K$ admits a crossing change at a crossing $c$ which preserves the oriented isotopy class of the knot, then $c$ is nugatory.
\end{conjecture}

\begin{lemma}\label{strongly}
    Let $K$ be a knot for which the cosmetic crossing conjecture is true. The conversion of a single strongly essential crossing of a knot diagram of $K$ to a virtual crossing gives rise to a virtual knot of virtual crossing number equal to 1.
\end{lemma}

\begin{proof}
Let $K_{\vec{\xi}}$ denote the diagram of a knot, of type $\kappa$. Let $c$ be a strongly essential crossing of $K_{\vec{\xi}}$.  After converting the crossing $c$ to virtual, the diagrammatic virtual crossing number of the resulting virtual knot, $VK_{\vec{\xi}}$, is 1, thus its virtual crossing number is less than or equal to 1. This virtual knot can be obtained as the closure of two knotoids, $k_1,k_2$ (depending on which arc is opened).  If the virtual crossing number of $VK_{\vec{\xi}}$ is 0, then the corresponding knotoids, $k_1,k_2$, are of knot-type and their over/under closures give the same knot, $\kappa$. That implies that the knot would be invariant under crossing change at the crossing c. Since the knot is assumed to satisfy the cosmetic crossing conjecture, the crossing is nugatory. By Corollary 2.41 in \cite{Detoffoli2013}, a nugatory crossing cannot be strongly essential, contradiction.

\end{proof}

\section{Open curve neighborhoods of knots}

In the following we consider the set $\Omega$ of oriented simple open and closed curves in 3-space that can be approximated by polygonal curves. For closed curves, these correspond to tame knots, we denote $TM$. Let $X\subset \Omega$ denote the space of simple open curves in 3-space. The boundary of $\Omega$  is then singular (open or closed) curves in 3-space. 
We will also be interested in a subset of $X$ which consists of open curves in 3-space, whose end-to-end closure arc has non-empty intersection with the curve, we denote $SX$.
We will particularly focus on a subset of $\Omega$, that of non-degerate polygonal approximations of open and closed curves in 3-space, we denote $\Omega_P$ \cite{Denne2006}:

\begin{definition}
    A non-degenerate polygonal curve $K$ is generic if it satisfies the following genericity conditions: (i) Given any three pairwise skew edges $e_1,e_2,e_3$ of $K$ and the doubly-ruled surface $H$ that they generate, no vertex of $K$ (except endpoints of $e_1,e_2,e_3$) may lie in $H$. (ii) There are no quintiscecants (or higher order secants). (iii) There are no vertex trisecant lines which lie in the osculating plane of the vertex.
\end{definition}

\noindent The set of non-degenerate polygonal curves of $n$ edges is an open dense set in $R^{3n}$ \cite{Denne2006}.

\begin{definition}
    For any two open piecewise linear curves in 3-space, $l_1,l_2 \in X$, defined by vertices $\lbrace x_1,\dotsc,x_n\rbrace$ and $\lbrace y_1,\dotsc,y_n\rbrace$, respectively, we define their distance to be their Hausdorff distance $D(l_1,l_2)=\max_{x\in l_1}\lbrace\min_{y\in l_2}\lbrace d(x,y)\rbrace\rbrace$.
\end{definition}

\subsection{Knotoid Spectrum of a simple open curve in 3-space}

In this section we discuss the knotoid spectrum of a simple open curve in 3-space and its properties.

The knot spectrum of an open curve in 3-space, originally defined in \cite{Millett2004,Millett2005} and applied to proteins in \cite{Sulkowska2012} was extended to the knotoid spectrum of a protein in \cite{Goundaroulis2017}. Below follows the definition of the knotoid spectrum and some similar definitions that stem from it:

\begin{definition}[Knotoid spectrum of an open curve in 3-space]
    Let $l\in X$ denote an open curve in 3-space. For any $\vec{\xi}\in S^2$ (except a set of measure zero), $l_{\vec{\xi}}$ is a knotoid. The set of knotoids that are associated with $l$ for all $\vec{\xi}\in S^2$ is the knotoid spectrum of $l$, we denote $kspec(l)$ (see Figure \ref{neighborhood}). Each knotoid in the spectrum of an open curve can be associated with a geometric probability of a projection of the curve giving that knotoid. We will call the set of knotoids of $l$ and their associated probabilities, the geometric spectrum, we denote $gspec(l)$. The virtual spectrum of $l$ is the set of virtual knots that appear as virtual closures of the knotoids in its knotoid spectrum, we denote $vspec(l)$. 
\end{definition}

\begin{remark}
    Similarly, we can also define the knot spectrum of a knot via the closure methods at the sphere at infinity, the theta curve spectrum or the strongly invertible knot spectrum of an open curve in 3-space. 
\end{remark}

\begin{lemma}\label{const}
    The knotoid spectrum of an open curve in 3-space is a finite set and it is locally constant in $X$.
\end{lemma}

\begin{proof}
   Let $l\in X$ be a piecewise linear open curve in 3-space of $n$ edges.  The probability of a particular knotoid in the spectrum occuring in a projection is defined by a spherical area whose boundary is formed by the intersection of great circles defined by the vertices and edges of $l$ \cite{Banchoff1976}. There are finitely many such great circles.  
   Due to continuity of the boundary of the spherical area on the coordinates of the curve, there is always small enough deformation so that the spherical areas corresponding to the distinct knotoids in the spectrum change but do not vanish or create intersections between new pairs of great circles.
\end{proof}

\begin{lemma}\label{knottypespec}
Let $l\in X$ be an open curve in 3-space, then

\noindent (i) If a projection of $l$ gives a knot-type knotoid, $K$, then this is the unique knot-type knotoid in the spectrum of $l$.

\noindent (ii) For all $l\in X\setminus SK$, $kspec(l)$ contains a knot type knotoid.

\end{lemma}

\begin{proof}

(i) Let $l_{\vec{\xi}_1}$ be a projection of $l$ that gives the knot-type knotoid $\kappa$, and let $\theta_1$ denote the corresponding $\theta$ curve. Suppose that there is another projection $l_{\vec{\xi}_2}$ that gives another knot-type knotoid, $\kappa_2$, and let $\theta_2$ denote its theta curve. Adjust the over/under arcs of $\theta_1$ (resp. $\theta_2$), so that the projection of the over/under arcs of $\theta_1$ (resp. $\theta_2$) lie in the same region as the projection of the endpoints of $l$ with respect to $\vec{\xi}_1\in S^2$ (resp. $\vec{\xi}_2\in S^2$) . Then $\theta_1, \theta_2$ are equivalent, so $\kappa_1,\kappa_2$ are equivalent.

\noindent (ii) Let $l\in X, l\in\Omega_P$ and suppose that $l\not\in SK$. Let $\vec{\xi}\in S^2$  be the unit vector in the direction of the vector that connects the endpoints, say $x_0,x_0'$, of $l$, and assume that $x_0,x_0'$ are not part of a higher order secant. The projection with respect to $\vec{\xi}$ gives a diagram, where the two endpoints of $l$ coincide and lie in the same region of the diagram. By Lemma \ref{const} there is another vector $\vec{\xi}'\in S^2$ in an $\epsilon$-neighborhood of $\vec{\xi}\in S^2$ for $\epsilon>0$ and small enough, such that both endpoints project in the same region of the diagram and do not coincide (since $\vec{\xi}'\neq\vec{\xi}$), resulting in a knot-type knotoid.  

If $l\in SK$, the end-to-end vector intersects $l$ at a point $x\in l$. Then in all projection directions, the end-to-end vector will cross the rest of the diagram at the projection of $x$. Thus, all projections will give a knotoid of diagrammatic height at least 1, so $kspec(l)$ may not have a knot-type knotoid. 

\end{proof}

\begin{theorem}\label{injective}
    Let $S$ denote the set of knotoids and let $P^0(S)$ denote the powerset of $S$, excluding the trivial (empty) set. Let $\phi:X\rightarrow P^0(S)$ denote a map from the set of simple open curves in 3-space to the powerset of knotoids. Then $\phi$ is not surjective, nor injective.
\end{theorem}

\begin{proof}

     By Lemma \ref{const}, there exist $l,l'\in X, l\neq l'$, such that $kspec(l)=kspec(l')$. 
     
     \noindent The set of knotoids that consists of only the trivial knotoid and the trefoil knot-type knotoid cannot be the knotoid spectrum of any open curve in 3-space, since by Lemma \ref{knottypespec},the knot-type knotoid in the spectrum of an open curve is unique.

\end{proof}

\subsection{Open curve neighborhood of a knot}

In this section we introduce the open curve neighborhood of a knot which gives rise to the knotoid spectrum of a knot. We explore how the pure knotoids in the knot spectrum of a knot can distinguish non-equivalent knots.

\begin{definition}\label{openneighborhood}(Open curve neightborhood of a knot)
Let $K$ denote a knot, seen as a simple piecewise linear closed curve in 3-space and let $x\in K$ be a vertex of $K$. Let us denote by $K_{x}$ the open curve obtained by deleting the edge incoming to $x$ (assuming an ordering of the vertices of $K$).  The based $h_x-$neighborhood of the knot embedding at $x$, in the set of simple open curves in 3-space defined as $N_{h_x,x}(K)=\lbrace{l\in X|D(l,K_{x})<h\rbrace}$, where $h_x>0$ is small enough such that $l$ has the same knotoid spectrum as $K_{x}$ and for which there exists an ambient isotopy from $K_{x}$ to $l$, not passing through the end-to-end closure arc at any time.  The $h-$neighborhood of the knot embedding $K$, $N_h(K)$, is the union of all the based neighborhoods of the knot embedding, ie. $N_h(K)=\cup\lbrace N_{h_x,x}|x\in K\rbrace$, where $0<h<h_x$ for all $x\in K$. Notice that $N_{h_x,x}(K)\subset N_h(K)$.
\end{definition}

\begin{figure}[ht!]
    \centering
\includegraphics[scale=0.3]{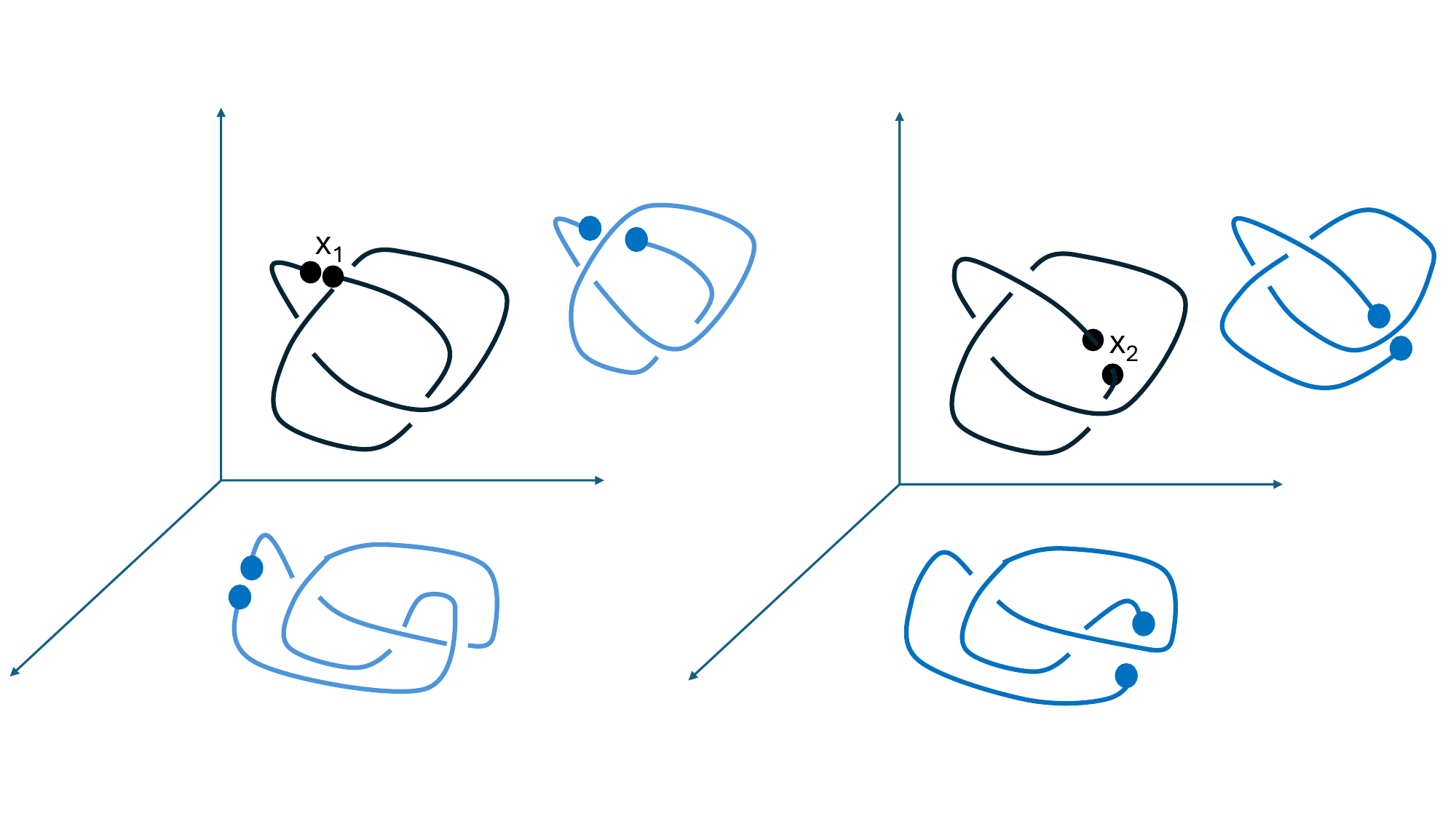} 
    \caption{Examples of open curves in the neighborhood of an embedding of the trefoil, $K$, based at different points, $x_1, x_2$ along the knot. Each open curve (shown in black), we denote $l_{x_1}, l_{x_2}$, gives rise to a knotoid spectrum (examples of knotoid diagrams shown in blue). The knotoids shown belong to the based neighborhoods of $K$ at $x_1$ and $x_2$, $N_{h,x_1}$ and $N_{h,x_2}$, respectively. The union of the knotoids from both belong to the neighborhood of $K$, $N_h(K)$.}
    \label{neighborhood}
\end{figure}

\begin{definition}[open curve neighborhood of a knot-type]
    Let $\kappa$ denote a knot type and let $X_{\kappa}=\cup_{K\sim\kappa} N_h(K)$, be the union of the neighborhoods of all embeddings of $\kappa$, which we call the open neighborhood of knot type $\kappa$. Then $X_{\kappa}\subset X$. Let us call $XK=\cup_{\kappa}X_{\kappa}$ the space of open knots and let its complement, $XU=X\setminus XK$, be called the space of transient knotting.
\end{definition}

\begin{definition}\label{knotspecknot}[Knotoid spectrum of a knot]
Let $K$ denote an embedding of a knot $\kappa$ and $x$ a point on $K$. We will define the knotoid spectrum of $K$ to be the union of the knotoid spectra of all the open curves in $N_{h}(K)$. We denote $kspec(K)=\cup_{l\in N_h(K)}\lbrace kspec(l)\rbrace$.  We will define the pure knotoid spectrum of $K$ to be all the knotoids in $kspec(K)$ except the knot-type knotoid $\kappa$, we denote $pkspec(K)$. We will denote the subset of knotoids of height $m\in N$ of the knotoid spectrum of an embedding of a knot, $H_m(kspec(K))$.  We define the geometric spectrum of a knot to be the spectrum with the associated geometric probabilities of each knotoid in the spectrum, we denote $gspec(K)=\cup_{l\in N_h(K)}\lbrace gspec(l)\rbrace$.   Similarly, we can define the virtual spectrum of the knot as the union of the virtual knotoid spectra of $K$, we denote $vspec(K)=\cup_{l\in N_h(K)}\lbrace vspec(l)\rbrace$.
\end{definition}

\begin{remark}
    We can also define the knot spectrum of a knot via closures, and the strongly invertible knot spectrum of a knot. 
\end{remark}

Notice that we can generalize the neighborhood of a knot to include collections of open curves in 3-space and associated linkoid spectra. The corresponding definitions are given in the Appendix. This will not be the focus of this manuscript.

The following Corollaries follow from Definition \ref{knotspecknot} and Lemma \ref{knottypespec}:

\begin{corollary}\label{knottype}
    Let $K$ denote an embedding of a knot of type $\kappa$. For all $l\in N_h(K)$, open curves in the neighborhood of $K$, the knotoid spectrum of $l$ contains the knot-type knotoid of type $\kappa$ and is the only knot-type knotoid in the spectrum. 
\end{corollary}

\begin{corollary}
     Let $K_1,K_2$ denote the embeddings of two knots in 3-space. $K_1$ and $K_2$ are equivalent knots if and only if the knot-type knotoids  in their spectra are equivalent.
\end{corollary}

\begin{lemma}
The following are true for the space $X$ and its knot subspaces.

\noindent (i) The space of open curve neighborhoods of embeddings of a given knot type $\kappa$, $X_{\kappa}$, is connected and dense.
    
\noindent (ii) Let $\kappa_1,\kappa_2$ be two non-equivalent knots, then $X_{\kappa_1}\cap X_{\kappa_2}=\emptyset$. 
\end{lemma}

\begin{proof}
    (i) $X_{\kappa}$ is dense as a union of open subsets of $X$. Any $K_0,K_1\in X_{\kappa}$, embeddings of the knot type $\kappa$ are related by ambient isotopy. Let $g:\mathbb{R}^3\times[0,1]\rightarrow \mathbb{R}^3$ be the isotopy from $K_0$ to $K_1$. Then $g_t\circ K_0=K_t$, where $K_t, t\in[0,1]$, is an embedding of the knot $\kappa$. Let $x\in K$ and consider the neighborhoods $N_t=N(h_{K_t},(K_t)_{g_t(x)})$ of $K_t$ based at $g_t(x)$, respectively. Let $0<h<h_{K_t}$ for all $t\in[0,1]$. By continuity of $g_t$ with respect to $t$, there is $\delta>0$ such that for all $t_1,t_2\in I$ such that $|t_1-t_2|<\delta$, then $D(g_{t_1}\circ K_0 ,g_{t_2}\circ K_0 )=D(K_{t_1},K_{t_2})<h/2$. Let $l\in N(h,K_{t_1})$, such that $D(l,K_{t_1})<h/2$. Since $l\in N(h,K_{t_1})$ then $l\simeq K_{t_1}$ via an isotopy, $h_t$, that does not cross the end-to-end closure arc. Then $F_t=h_{1-2t}, t\in[0,1/2]$ and $F_t=g_{(2-2t)t_1+(2t-1)t_2}, t\in[1/2,1]$ is an isotopy from $l$ to $K_{t_2}$, that does not cross the end-to-end closure arc of $F_t(l)$ for any $t$. Also $D(l,K_{t_2})\leq D(l,K_{t_1})+D(K_{t_1},K_{t_2})<h$.  It follows that $l\in N(h,K_{t_2})$, thus $l\in N(h,K_{t_1})\cap N(h,K_{t_2})$. 

\noindent  (ii) Let $\hat{\kappa_1}, \hat{\kappa_2}$ denote the knot type knotoids of knot type $\kappa_1, \kappa_2$, respectively. Suppose $l\in X_{\kappa_1}\cap X_{\kappa_2}$, then $\hat{\kappa_1},\hat{\kappa_2}\in kspec(l)$, contradiction by Lemma \ref{knottypespec}.

\end{proof}

\begin{proposition}\label{diag_height}
Let $K$ denote a knot embedding and let $x_0\in K$. Suppose that $K$ has $k_n$ $n$-secants, $n>2$, containing $x_0$. There exists $h$ small enough, such that for any $l\in N_{h,x_0}(K)$, open curve $l$ on the open neighborhood of $K$, based at $x_0$, its knotoid spectrum $kspec(l)$ consists of one knot type knotoid, knotoids of height less than or equal to 1, and exactly $k_n$  knotoids of height less than or equal to $n-1$, where $n>2$. If in addition $K$ is a generic polygon, $K\in\Omega_P$, then its knotoid spectrum consists of knotoids of only up to height 3.
\end{proposition}

\begin{proof}

By taking a projection with respect tona unit vector, say $\vec{\xi}\in S^2$, in the direction of an $n$-secant that goes through $x$, $K$ will have a non-generic projection where $n$ points coincide. The same projection direction of any $l\in N_{h,x_0}(K)$ will lead to a non generic projection with $n-1$ points projecting in one. Then, for small enough $\epsilon$, any vector in an $\epsilon$-neighborhood of $\vec{\xi}$, say $\vec{\xi}_1$, will give rise to a knotoid diagram. The closure arc will give rise to a corresponding virtual knot diagram with $n-1$ virtual crossings, thus, a virtual knot of less than or equal to $n-1$ virtual crossings or a knotoid of height less than or equal to $n-1$. For a generic polygon, there will be only up to quadrisecants \cite{Denne2006}, thus giving only up to height 3 knotoids.

\end{proof}

\begin{definition}[Gordian distance]
    The Gordian distance of two knots is the minimum number of crossing changes needed to go from one to the other, we denote $d_G(\kappa_1,\kappa_2)$ for two knots $\kappa_1,\kappa_2$ \cite{Murakami1985}. If one of the knots is the unknot, then this is the unknotting number of the non-trivial knot.
\end{definition}

\begin{proposition}\label{diagrammatic}
   
    Let $\kappa$ be a non-trivial knot of unknotting number $u(\kappa)$. Then for any generic embedding $K\in\Omega_P$ of $\kappa$ of $n$ edges, $kspec(K)$ contains at least $2u(\kappa)^2$  and at most $n/12(n-3)(n-4)(n-5)$ knotoids of diagrammatic height 3 in its spectrum, at least one is alternating. 
\end{proposition}

\begin{proof}
    It follows by \cite{Pannwitz1933,Kuperberg2003}, since every such knot has at least $2u(\kappa)^2$  and at most $n/12(n-3)(n-4)(n-5)$ quadrisecants and a projection of any $l\in N_h(K)$ based at one of the points in these quadrisecants,  will give knotoids of diagrammatic height 3. By \cite{Denne2006}, at least one will be alternating (see Figure \ref{quadr}).
\end{proof}

\begin{figure}[ht!]
    \centering
\includegraphics[scale=0.3]{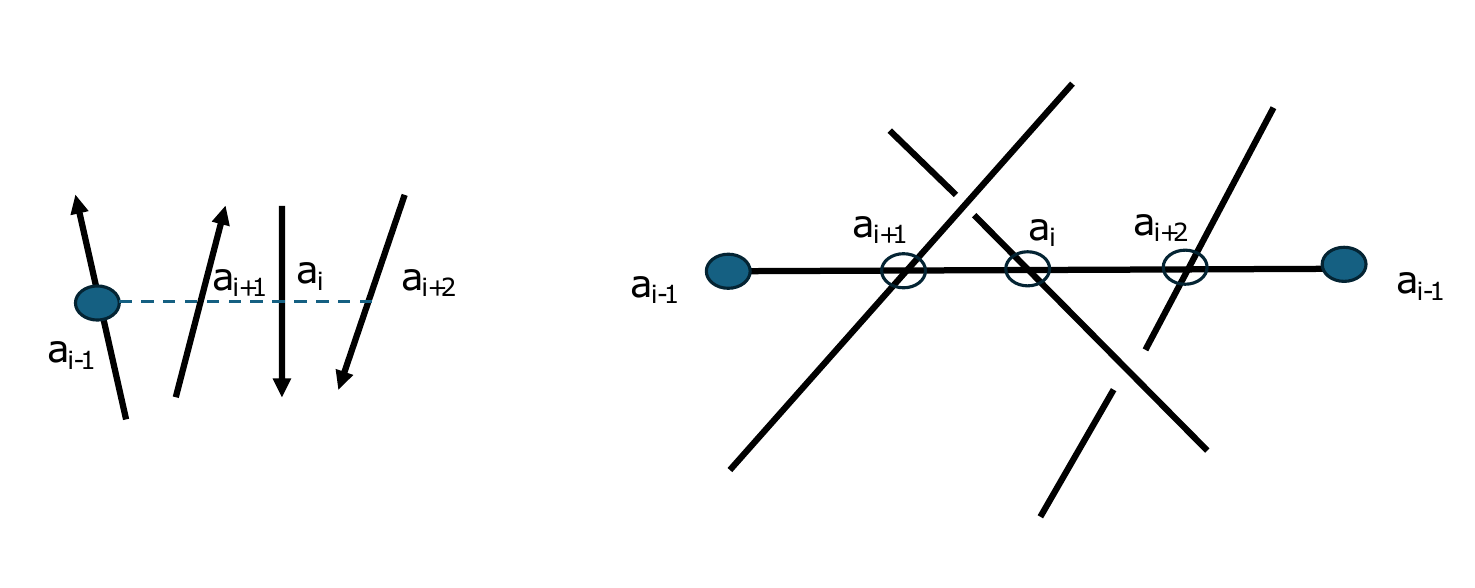} 
    \caption{(Left) An alternating quadtisecant that connects the points $a_{i-1}a_ia_{i+1}a_{i+2}$ on a piecewise linear knot embedding in 3-space (only 4 edges of the knot shown, with orientation). (Right) A part of a projection of an open curve in the based neighborhood of the corresponding knot based at $a_{i-1}$, in the direction of the quadrisecant of the knot. For simplicity, we denote by $a_{i-1}$ both endpoints of the edge incident to the point $a_{i-1}$. The virtual arc closure of the knotoid shown. The neighborhood of the knot based at $a_{i-1}$ gives rise to knotoids of diagrammatic height 3.}
    \label{quadr}
\end{figure}

\begin{theorem}\label{uncrossing}
   If $\kappa_1, \kappa_2$ are two knot types with $d_G(\kappa_1,\kappa_2)>n$, then for any two embeddings $K_1$ and $K_2$ of $\kappa_1,\kappa_2$, the knotoids of height $m$, $m\leq n$, in their spectra are all different. If $K_1,K_2\in\Omega_P$, and $d_G(\kappa_1,\kappa_2)>3$, then all knotoids in their spectra are different.
\end{theorem}

\begin{proof}
    Let $K_1,K_2$ be two embeddings of $\kappa_1, \kappa_2$. If $K_1,K_2$ have any common height $n$ knotoids, then $K_1,K_2$ differ by at most $n$ crossing changes, which would imply $d_G(\kappa_1,\kappa_2)\leq n$, contradiction. By Proposition \ref{diagrammatic}, for $K_1,K_2\in\Omega_P$ there are knotoids of only up to height 3, thus if $d_G(K_1,K_2)>3$, $kspec(K_1)\cap kspec(K_2)=\emptyset$. 
\end{proof}

\begin{corollary}
    Let $\kappa$ be a non trivial knot and let $u(\kappa)$ denote its unknotting number. Then for any generic embedding $K$ of $\kappa$ and any generic embedding $U$ of the unknot:

\noindent (i) If $u(\kappa)>3$, then $kspec(K)\cap kspec(U)=\emptyset$.

\noindent    (ii) If $u(\kappa)\geq j$, $j=2,3$, then $H_m(kspec(K))\cap H_m(kspec(U))=\emptyset$ for all $m<j$. 

\noindent    (iii) If $u(\kappa)=1$, then $k\in H_1(kspec(K))\cap H_1(kspec(U))$, if and only if $k$ is a knotoid that occurs from $K$ when projected with respect to a strongly essential secant that corresponds to a crossing change that would give the unknot.

\end{corollary}

\begin{proof}
    (i) and (ii) follow from Theorem \ref{uncrossing}. (ii) If $k\in H_1(kspec(K))\cap H_1(kspec(U))$, then the over and under closure of $k$ gives $K$ and $U$. Thus, the corresponding secant of $K$ determines a crossing whose change in $K$ gives the unknot. Such a secant is strongly essential by Corollary 2.48 in \cite{Detoffoli2013}.
\end{proof}

\begin{theorem}\label{Xk}
    The pure knotoid spectrum of knots can distinguish the space of unknots, $X_{o}$, from the space $X_{\kappa}$ of any non-trivial knot $\kappa$ that satisfies the cosmetic crossing conjecture.
\end{theorem}

\begin{proof}
    There exist embeddings of the unknot at the neighborhood of which the knotoid spectrum consists of only the trivial knotoid, such as generic embeddings of the circle obtained as $h$-deformations of the planar circle. On the other hand, any embedding of a non-trivial knot contains at least one strongly essential secant (by Theorem 2.31 in \cite{Detoffoli2013}). By Lemma \ref{strongly}, if this is a knot that satisfies the cosmetic crossing conjecture, then this gives rise to a height 1 knotoid in the knotoid spectrum of the knot.
\end{proof}

\section{Topological invariants at the neighborhood of knots}

By \cite{Panagiotou2021} we obtain a general framework by which invariants of knots can be rigorously defined for open curves in 3-space to give continuous functions of the curve coordinates that tend to topological invariants when the endpoints tend to coincide to form a knot.

\begin{definition}
    Let $f$ denote an invariant of knots and knotoids and let $l$ denote an open curve in 3-space. We define an $f$-measure of entanglement of open curves in 3-space

    \begin{equation}
        f(l)=\frac{1}{4\pi}\int_{\vec{\xi}\in S^2}f(l_{\vec{\xi}})dA
    \end{equation}
\end{definition}

\noindent We can also express $f$ as $f(l)=\sum_{i\in kspec(l)}p_i(h)f(l_i)$, where $l_i$ are  knotoids that appear in the projections of $l$ and where $p_i(h)$ is the geometric probability that a projection of $l$ gives the knotoid $l_i$.

Then $f$ is a continuous function of the coordinates of $l$. If the value of $f$ for a knot-type knotoid is that of the knot type, then $f$ converges to the topological invariant of the resulting knot, as the endpoints of $l$ tend to coincide to form a knot \cite{Panagiotou2021}.

In the following of this section let $f$ be a measure of complexity that is a continuous function of the curve coordinates in $X$, and topological invariant when restricted to knots and an invariant of knotoids.

\begin{definition}
    Let $f$ be an invariant of knots as above. Let $l\in X$, open simple curve in 3-space, then $fkspec(l)=\lbrace{f(k_i), k_i\in kspec(l)\rbrace}$ denotes the spectrum of $f$, which is the set of values of $f$ on the knotoids in the knotoid spectrum of $l$. Let $K$ denote an embedding of a knot $\kappa$. We define the spectrum of $f(K)$ to be the set of values of $f$ in the neighborhood of the knot, ie $f(N_h)=\lbrace f(l)|l\in N_h(K)\rbrace$. We define the knotoid spectrum of $f(K)$ to be $fkspec(K)=\lbrace{fkspec(l)|l\in N_h(K)\rbrace}$. 
\end{definition}

\begin{remark}
    Similarly, one can define the maximum/minimum/average of $f$ in the spectrum of $K$. We can define with analogous properties the virtual and geometric spectra $fvspec(K), fgspec(K)$, as well as for the knot spectrum of a knot.
\end{remark}

\begin{corollary}\label{fdg}
   Let $f$ denote an invariant of knots and  complete invariant of pure knotoids and let $\kappa_1, \kappa_2$ be two knot types and suppose $f(\kappa_1)=f(\kappa_2)$. If $d_G(\kappa_1,\kappa_2)>n$, then for any two embeddings of $\kappa_1,\kappa_2$, $f$ can distinguish all the pure knotoids of height up to $n$ in their spectra.  If $K_1,K_2\in\Omega_P$, $d_G(\kappa_1,\kappa_2)>3$, then $fpkspec(K_1)\cap fpkspec(K_2)=\empty$.
\end{corollary}

\begin{corollary}
    Let $f$ denote an invariant of knots and knotoids. Let $\kappa$ be a non trivial knot and let $u(\kappa)$ denote its unknotting number and let $o$ denote the unknot. Suppose that $f(\kappa)=f(o)$, and that $f$ is a complete invariant of  knotoids. Then if $u(\kappa)>1$, for any generic embedding $K$ of $\kappa$ and any generic embedding $O$ of the unknot, $f$ can distinguish the two by its values in their knotoid spectra. 
\end{corollary}

\begin{theorem}\label{h1}
    Let $f$ be a topological invariant that distinguishes all height one knotoids from the trivial knotoid. Then $fkspec$ in the space of knots differentiates the space of unknots from that of any non-trivial knot that satisfies the cosmetic crossing conjecture.
\end{theorem}

\begin{proof}
    As in the proof of Theorem \ref{Xk}, there exist embeddings of the unknot at the neighborhood of which the knotoid spectrum consists of only the trivial knotoid, which has height zero. On the other hand, the knotoid spectra of the neighborhood of any embedding of a non-trivial knot contains knotoids of height 1. 
\end{proof}

As an example of the applicability of this result, we point out the following Corollary:

\begin{corollary}
    The Jones polynomial can distinguish the neighborhoods of at least some embeddings of the unknot from the neighborhoods of any embedding of any non-trivial knot if it can distinguish all height one knotoids from the trivial knotoid.
\end{corollary}

\subsection{Topological invariants in the neighborhood of mutant knots}

\begin{definition}[Concatenation of open curves in 3-space, 3D-tangle]
    For two simple open curves in 3-space, $l_1, l_2$ we define their concatenation by translating the one curve to begin when the other ends. 
Notice that with probability 1, this gives a simple open curve in $\mathbb{R}^3$.
In the special case that $l_1,l_2$ are each contained in a sphere with one endpoiont on the boundary each, we denote their concatenation as $l=l_1\# l_2$. We call a 3D tangle a collection of open curves in 3-space such that the endpoints of the curves touch a sphere that encloses them. If the endpoints of $l_1, l_2$ lie in a sphere so that when concatenated they form a closed loop, then the result is the connected sum of the end-to-end closure on the sphere of $l_1$ and $l_2$. 
\end{definition}

\begin{definition}(mutant embedding)
    Let $l$ be an open curve in 3-space such that $l=l_1\#l_T$, where $l_T$ is an open curve 3D tangle. We call two open curves in 3-space mutants, if one is obtained by the other by rotation and regluing of $l_T$. We denote $m(l)=l_1\#l_{m(T)}$. Similarly, let $K$ be an embedding of a knot $\kappa$ such that $K=l_T\cup l_1$, is the union of two open curve tangles such that $l_T$ is contained in a sphere. Then the mutant embedding of $K$ is $m(K)=l_{M(T)}\cup l_1$, where $l_{M(T)}$ is the rotation of the open curve tangle $l_T$ inside the sphere.
\end{definition}

\begin{remark}
In the above, we require that the sphere is a geometric sphere and not topological.    
\end{remark}

\begin{figure}[ht!]
    \centering
\includegraphics[scale=0.15]{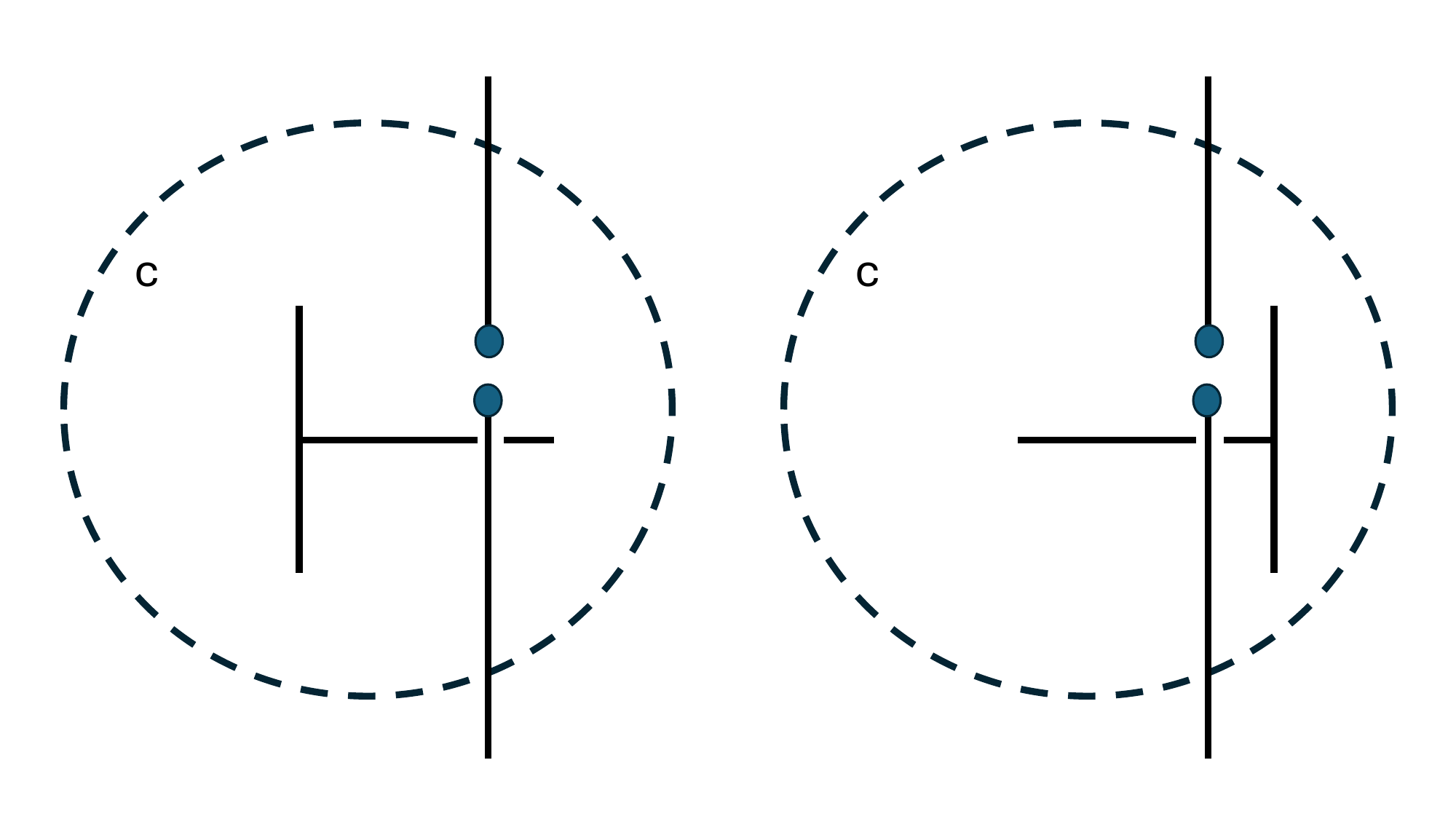} 
    \caption{Parts of projections of two mutant open curves in the neighborhood of two mutant knot embeddings. Outside the circle, both projections are identical. Inside the circle, the two diagrams may not be related by rotation.}
    \label{mutants}
\end{figure}

\begin{theorem}
    Let $K=l_1\#l_T$ be an embedding of a knot $\kappa$ and let $m(K)=l_1\#l_{M(T)}$ be a mutant embedding of $K$. Let $f$ denote an invariant of knots that cannot distinguish mutant knots and mutant knotoids. Then $fpkspec(K)\neq fpkspec(m(K))$, unless $f$ also fails to distinguish non-mutant knotoids that appear in their spectra.
\end{theorem}

\begin{proof}

Consider a piecewise linear embedding of a knot $K$ that contains a tangle $K=l_1\#l_T$, where $l_1,l_T$ are collections of open curves. 
Let $l\in N_x(K)$, where $x\in l_1$, then $l=l_1^{\epsilon}\#l_T$, where $l_1^{\epsilon}$ contains the endpoints of $l$. Let the mutant of $K$ be 
$m(K)=l_1\#l_{M(T)}$ and let $m(l)\in N_x(m(K))$, then  $m(l)=l_1^{\epsilon}\#l_{M(T)}$.
Let $S$ denote the sphere that contains $l_T$ (and also $l_{M(T)}$). In any projection direction, the projection of $S$ is a circle, we denote $c$, and the projection of the tangle is contained within $c$. Let us consider a projection with respect to a vector $\vec{\xi}\in S^2$. Then $(l_T)_{\vec{\xi}}$ and $(l_{M(T)})_{\vec{\xi}}$ are inside $c$ but $(l_1)_{\vec{\xi}}$ may overlap with the projection of $c$. Thus, we can express $(l)_{\vec{\xi}}=k_1\#k_2$, (resp. $(m(l))_{\vec{\xi}}=k_1\#k_2'$) where $k_1$ is outside $c$ and $k_2$ (resp. $k_2'$) is in $c$. Thus $k_2$ (resp. $k_2'$) is a linkoid that consists of $(l_T)_{\vec{\xi}}$ (resp. $(l_{M(T)})_{\vec{\xi}}$) and arcs of $l_1$ that project inside $c$. Let us take $\vec{\xi}\in S^2$ such that the endpoints of $l_1^{\epsilon}$ project inside $c$, thus $k_2$ (resp. $k_2'$) contains the endpoints. Notice that since $l_1$ does not intersect $S$, its arc that contains $x$ projects either above $(l_T)_{\vec{\xi}}$ and $(l_{M(T)})_{\vec{\xi}}$ or below. Thus, the arcs that contain the endpoints of $k_2$ (resp. $k_2'$) are either both above or both below $(l_T)_{\vec{\xi}}$ (resp. $(l_{M(T)})_{\vec{\xi}}$). Then $l_{\vec{\xi}}=k_1\# (l_T)_{\vec{\xi}}'$ and $l_{\vec{\xi}}=k_1\# ((l_{M(T)})_{\vec{\xi}})'$  where $(l_T)_{\vec{\xi}}', ((l_{M(T)})_{\vec{\xi}})'$ are not in general mutant knots (see Figure \ref{mutants}).

\end{proof}

As an example of the applicability of this result, we point out the following Corollary:

\begin{corollary}
    The Jones polynomial can distinguish the open curve neighborhoods of mutant embeddings of the Conway and Kinoshita-Terasaka knots, unless it fails to distinguish non-mutant pure knotoids in their spectra.  
\end{corollary}

This leads to the following conjecture:

\begin{conjecture}
    The Jones polynomial can distinguish the open curve neighborhoods of any two embeddings of the Conway and Kinoshita-Terasaka knots.  
\end{conjecture}

\begin{remark}
    The idea of studying the knot spectrum of a knot to distinguish the Kinoshita-Terasaka and Conway knots was proposed to Ken Millett and aspects of it were explored experimentally via the HOMFLY-PT polynomial in \cite{KTM}. 
\end{remark}

\subsection{The rate of change of an invariant in the neighborhood of a knot}

Notice that $f$ is differentiable in the neighborhood of a knot.
Let $K$ denote a knot embedding of a knot $\kappa$ and $l\in N(h,K_x)\subset N_h(K)$, both piecewise linear of $n$ vertices. Let $\vec{v}$ denote the vector that denotes the ordered difference of vertices in $\mathbb{R}^n$. Let $\vec{v}=K-l$ denote the vector in $\mathbb{R}^{3n}$ formed by subtracting the coordinates of the $i$th vertex of $K$ from that of $l$. The rate of change of $f$ at $K$, in the direction of $l$, $\vec{v}$, is $Df_{\vec{v}}(K)=D_{\vec{v}}p_0(K)f(\kappa)+\sum_{i\in pkspec(l)}D_{\vec{v}}p_i(K)f(l_i)$, where $l_i$ are the pure knotoids that appear in the spectrum of $l$ and $D_{\vec{v}}p_i(K)$ is the rate of change in the direction of $\vec{v}$ of the geometric probability that a projection of $l$ gives the knotoid $l_i$ and $p_0(K)$ is the geometric probability that $l$ gives the knot-type knotoid of type $\kappa$.

\begin{proposition}
   
    Let $f$ be a topological invariant that distinguishes all height one knotoids from any knot-type knotoid. The rate of change of $f$ in the space of knots differentiates the space of unknots from that of any non-trivial knot that satisfies the cosmetic crossing conjecture.
\end{proposition}

\begin{proof}
    Following the proof of Theorem \ref{h1}, there exist embeddings of the unknot at the neighborhood of which the rate of change of the unknot is 0, but there is no such neighborhood for any ambedding of any non-trivial knot, as any embedding of a non-trivial knot that satisfies the cosmetic crossing conjecture contains at least one strongly essential secant that gives rise a height 1 knotoid. If $f$ can distinguish all height one knotoids from all knot-type knotoids, then the rate of change of $f$ is non-zero. 
\end{proof}

\section{Conclusion}

In this manuscript, we introduce a novel framework for studying knots by studying the simple open curves in their neighborhood, which capture both geometrical and topological information about the knot.

We prove that the open neighborhood of a knot can be associated to a knotoid spectrum and can identify the knot type via the knot-type knotoid in the spectrum. More importantly, we prove that, at least for some knots, the pure knotoids in their spectra can fully distinguish them. Namely, we prove that for any two embeddings of any two knots that differ by more than one crossing change, at  some pure knotoids of specific height in their spectra are different. Moreover, for knots with Gordian distance greater than 3, and for any of their generic embeddings, all the pure knotoids in the knotoid spectra are different. We also prove that at least for some embeddings of the unknot, the pure knotoids in their pure neighborhood knotoid spectra are different than those of any embedding of any non-trivial knot that satisfy the cosmetic crossing conjecture.

Via a method that was introduced in \cite{Panagiotou2020b}, invariants of knots can be extended to be defined to open curves in 3-space and therefore also to the open curve neighborhood of knots.  These are continuous functions that tend to the topological invariant at the knot. We prove that invariants of knots that fail to distinguish certain knots, may be able to distinguish them by their values at their neighborhoods. For example, we prove that an invariant that cannot distinguish mutant knots and mutant knotoids, may be able to distinguish the neighborhoods of certain embeddings of mutant knots via the knotoid spectra, unless it also fails to distinguish non-mutant knotoids in their spectra.  We also prove that there are embeddings of the unknot for which the rate of change of a topological measure at their neighborhood is zero, while for any embedding of a non-trivial knot that satisfies the cosmetic crossing conjecture, the rate of change is non-zero, unless it fails to distinguish height one knotoids in their spectra from knot-type knotoids. 

The method presented in this manuscript is general and can be adapted to derive from the neighborhood of a knot, a knot spectrum, a strongly invertible knot spectrum, a linkoid or link spectrum, etc, thereby enabling many well studied areas in knot theory to be applied to a knot neighborhood to capture more information about a knot. 


\section{Appendix A: Linkoid spectrum of a knot}

In this section we discuss the definition of the linkoid spectrum of a knot. 

A linkoid is a multi-component knotoid. Linkoids are also connected to virtual knots via closure of the components. However, a virtual closure of linkoids requires a specification of the endpoints that are connected. This has been explored in detail in \cite{Barkataki2024}.
    
\begin{definition}\label{openneighborhood2}
Let $K$ denote a piecewise linear simple closed curve in 3-space (the embedding of a knot) and let $\textbf{x}=\lbrace{x_1,\dotsc,x_m\rbrace}\in K$, $m$ vertices on the knot. Let us denote by $K_{\textbf{x}}$ the collection of open curves obtained by deleting the edges incoming to $x_1,\dotsc,x_m$, respectively.  The based $h_{\textbf{x}}-$neighborhood of the knot embedding at $x_1,\dotsc,x_m$, in the space of simple curves in 3-space defined as $N_{h_{\textbf{x}},\textbf{x}}(K)=\lbrace{l\in X|D(l,K_{\textbf{x}})<h\rbrace}$, where $l$ is a collection of open curves in 3-space, $h_{\textbf{x}}$ is the minimum value such that $l$ has the same linkoid spectrum as $K_{\textbf{x}}$ and for which there exists an ambient isotopy from $K_{x}$, leaving its endpoints fixed and not passing through the corresponding closure arcs.  The $h-$neighborhood of the knot embedding $K$, $N_h(K)$, is the union of all the based neighborhoods of the knot embedding, ie. $N_h(K)=\cup\lbrace N_{h_{\textbf{x}},\textbf{x}}|\textbf{x}=\lbrace{x_1,\dotsc,x_m\rbrace}\in K\rbrace$, where $0<h<h_{\textbf{x}}$. Notice that $N_{h_{textbf{x}},\textbf{x}}(K)\subset N_h(K)$. 

Let $K$ denote an embedding of a knot $\kappa$ and $x_1,\dotsc,x_m$ points on $K$. We will define the linkoid spectrum of $K$ to be the union of the linkoid spectra of all the collections of open curves in $N_{h}(K)$. We denote $lspec(K)=\cup_{l\in N_h(K)}\lbrace lspec(l)\rbrace$.  
\end{definition}


\section{Acknowledgements}

The author acknowledges support by the National Science Foundation (Grant No. DMS-1913180 and NSF CAREER 2047587).


\bibliographystyle{plain} 
\bibliography{main}

\end{document}